\documentclass{article}

\PassOptionsToPackage{numbers, compress}{natbib}

\usepackage[final]{neurips_2022}
\usepackage[numbers]{natbib}

\usepackage[utf8]{inputenc} 
\usepackage[T1]{fontenc}    
\usepackage[hidelinks]{hyperref}       
\usepackage{url}            
\usepackage{booktabs}       
\usepackage{amsfonts}       
\usepackage{nicefrac}       
\usepackage{microtype}      
\usepackage{xcolor}         
\usepackage{float}
\usepackage{geometry}                		
\usepackage{graphicx}	
\usepackage{subcaption}
\usepackage{amssymb}
\usepackage{amsthm}

\usepackage[ruled,vlined]{algorithm2e}
\usepackage{enumerate}
\usepackage{longtable}
\usepackage{listings}

\newcommand{\R}{{\bf R}}

\newcommand{\eye}{\mbox{\bf I}}

\usepackage{physics}

\newtheorem{theorem}{Theorem}
\newtheorem{lemma}{Lemma}
\newtheorem{fact}{Fact}

\newtheorem{definition}{Definition}

\SetKwRepeat{Do}{do}{while}

\usepackage{accents}
\newcommand{\ubar}[1]{\underaccent{\bar}{#1}}

\newcommand{\edit}[1]{{\color{blue}#1}}

\renewcommand{\edit}[1]{#1}

\newcommand{\N}{\mathbf{N}}

\newcommand{\dmax}{\bar{d}_{\epsilon}}
\newcommand{\dmin}{\ubar{d}_{\epsilon}}
\newcommand{\sm}{S}

\newcommand{\SPTone}{\gamma_1}
\newcommand{\SPTtwo}{\gamma_2}
\newcommand{\SPTthree}{\gamma_3}

\newcommand{\diam}{\mathbf{diam}}

\newcommand{\NBH}{N}

\newcommand{\RequireSPT}{$\SPTone \in [0, 1)$, $\SPTtwo \in (1/\omega, 1]$, $\SPTthree \in (0,1]$}

\usepackage{pifont}
\usepackage{tablefootnote}

\title{A consistently adaptive trust-region method}

\author{%
  Fadi Hamad \\
  Department of Industrial Engineering\\
  University of Pittsburgh\\
  Pittsburgh, PA 15261 \\
  \texttt{fah33@pitt.edu} \\
   \And
   Oliver Hinder \\
    Department of Industrial Engineering\\
   University of Pittsburgh \\
   Pittsburgh, PA 15261 \\
   \texttt{ohinder@pitt.edu} \\
}

\begin{document}

\maketitle

\begin{abstract}
\edit{
 Adaptive trust-region methods attempt to maintain strong convergence guarantees without depending on conservative estimates of problem properties such as Lipschitz constants. However, on close inspection, one can show existing adaptive trust-region methods have theoretical guarantees with severely suboptimal dependence on problem properties such as the Lipschitz constant of the Hessian. For example, TRACE developed by Curtis et al. obtains a $O(\Delta_f L^{3/2} \epsilon^{-3/2}) + \tilde{O}(1)$ iteration bound where $L$ is the Lipschitz constant of the Hessian. Compared with the optimal $O(\Delta_f L^{1/2} \epsilon^{-3/2})$ bound this is suboptimal with respect to $L$. We present the first adaptive trust-region method which circumvents this issue and requires at most $O( \Delta_f L^{1/2}  \epsilon^{-3/2}) + \tilde{O}(1)$ iterations to find an $\epsilon$-approximate stationary point, matching the optimal iteration bound up to an additive logarithmic term. Our method is a simple variant of a classic trust-region method and in our experiments performs competitively with both ARC and a classical trust-region method.}
\end{abstract} 

\section{Introduction}

Second-order methods are known to quickly and accurately solve sparse nonconvex optimization problems
that, for example, arise in optimal control \cite{betts2010practical}, truss design \cite{schittkowski1994numerical}, AC optimal power flow \cite{frank2012primer}, and PDE constrained optimization \cite{biegler2003large}. 
Recently, there has also been a large push to extend second-order methods to tackle machine learning problems by coupling them with carefully designed subproblem solvers \cite{kohler2017sub, zhou2018stochastic, lee2018distributed,curtis2021worst,curtis2020fully,curtis2019stochastic,gupta2021localnewton,xu2020newton,fang2020newton,wang2018giant,na2022hessian,xu2016sub, crane2020dino, zhang2022adaptive,kovalev2019stochastic}. 

Much of the early theory for second-order methods focused on showing fast local convergence and (eventual) global convergence \cite{gill1974newton,more1979use,sorensen1982newton,ulbrich2004superlinear,vicente2002local,byrd2000trust, chen2006interior, conn2000primal, gould2002interior, wachter2005line}. These proofs of global convergence, unsatisfactorily, rested on showing at each iteration second-order methods reduced the function value almost as much as gradient descent \cite[Theorem 4.5]{sorensen1982newton} \cite[Theorem 3.2.]{wright2006numerical}, this is despite the fact that in practice second-order methods require far fewer iterations. In 2006, Nesterov and Polyak \cite{nesterov2006cubic} partially resolved this inconsistency by introducing a new second-order method, cubic regularized Newton's method (CRN).
Their method can be used to find stationary points of multivariate and possibly nonconvex functions $f : \R^{n} \rightarrow \R$.
Their convergence results assumes the optimality gap
\[ 
\Delta_f := f(x_1) - \inf_{x \in \R^{n}} f(x)
\]
is finite and that the Hessian of $f$ is $L$-Lipschitz: 
\begin{flalign}
\| \grad^2 f(x) - \grad^2 f(x') \| \le L \| x - x' \| \quad \forall x, x' \in \R^{n} 
\end{flalign}
where $\| \cdot \|$ is the spectral norm for matricies and the Euclidean norm for vectors.
If $L$ is known they guarantee their algorithm terminates with an $\epsilon$-approximate stationary point:
\[ 
\| \grad f(x) \| \le \epsilon
\] 
after at most
\begin{flalign}\label{CRN:iteration-bound}
O(\Delta_f  L^{1/2} \epsilon^{-3/2})
\end{flalign}
iterations. For sufficiently small $\epsilon$, this improves on the classic guarantee that gradient descent terminates after at most $O(\Delta_f   \sm \epsilon^{-2})$ iterations for $\sm$-smooth functions, thereby partially resolving this inconsistency between theory and practice.
Bound \eqref{CRN:iteration-bound} is also known to be the best possible for second-order methods \cite{carmon2017loweri}.

However, CRN only achieves \eqref{CRN:iteration-bound} if the Lipschitz constant of the Hessian is known.
In practice, we rarely know the Lipschitz constant of the Hessian, and if we do it is likely to be a conservative estimate. With this in
mind, many authors have developed
practical algorithms that achieve the convergence guarantees of CRN without 
needing to know the Lipschitz constant of the Hessian. 
\edit{We list these adaptive second-order methods in Table~\ref{table:adaptive-second-order-methods} along with their worst-case iteration bounds.
}

\edit{
Despite the fact that all these algorithms match the $\epsilon$-dependence of $\eqref{CRN:iteration-bound}$, the majority of them are suboptimal due to the dependency on the Lipschitz constant $L$. For example, only our method and \cite{grapiglia2017regularized,cartis2019universal} are optimal in terms of $L$ scaling. Whereas \cite{cartis2011adaptiveII} is suboptimal as the bound scales proportional to $L ^ {3/2}$ instead of $L ^ {1/2}$. Moreover, all the trust-region methods have suboptimal $L$ scaling. In particular, inspection of these bounds shows scaling with respect to $L$ of $L ^ {3/2}$ for \cite{curtis2017trust}
and $L ^{2}$ for \cite{curtis2021trust}
instead of the optimal scaling of $L ^ {1/2}$. 
}

An ideal algorithm wouldn't incur this cost for adaptivity. This motivates the following definition.

\edit{ 
\begin{definition}
A method is \textbf{consistently adaptive} on a problem class if, without
 knowing problem parameters, it achieves
the same worst-case complexity bound as one obtains if problem parameters were known, up to a \textbf{problem-independent} constant-factor and
additive polylogarithmic term.
\end{definition}
}

\edit{
 Clearly, based on our above discussion there does not exist consistently adaptive trust-region methods. Indeed, despite the extensive literature on trust-region methods \cite{curtis2020fully,curtis2019stochastic,sorensen1982newton,byrd2000trust,conn2000primal,curtis2017trust,conn2000trust,powell1984global,carter1991global,powell2003trust} and their worst-case iterations bounds \cite{curtis2017trust,curtis2021trust,curtis2018concise}, none of these methods are consistently adaptive. 
  As we mentioned earlier and according to Table \ref{table:adaptive-second-order-methods}, \cite{grapiglia2017regularized,cartis2019universal} are cubic regularization based methods which scale optimally with respect to the problem parameters. However, they are not quite consistently adaptive because $\sigma_0$ appears outside the additive polylogarithmic term.

    \begin{table}[!thb]
    \caption{Adaptive second-order methods along with their worst-case  bounds on the number of gradient, function and Hessian evaluations. $\sigma_{\min}\in (0,\infty)$ is the smallest regularization parameter used by ARC \cite{cartis2011adaptiveII}.
    $\sigma_0 \in (0,\infty)$ is the initial regularization parameter for cubic regularized methods.
    }
    \label{table:adaptive-second-order-methods}
    \centering
    {\renewcommand{\arraystretch}{1.3}
    \begin{tabular}{lll}
            \hline
            \textbf{Algorithm} & \textbf{type} & \textbf{worst-case iterations bound} \\\midrule
            \textbf{ARC \cite{cartis2011adaptiveII}} \tablefootnote{
	Obtaining this bound does require carefully inspection of 
	Cartis, Gould and Toint \cite{cartis2011adaptiveII} (who highlighted only on the $\epsilon$-dependence of their bound).
	For simplicity of discussion we assume the ARC subproblems are solved exactly (i.e., $C=0$, $\kappa_\theta = 0$), and that the initial regularization parameter satisfies $\sigma_0 = O(L + \sigma_{\min})$ (the bound only gets worse otherwise). We also consider only the bound on the number of Hessian evaluations, inclusion of the unsuccessful iterations (where cubic regularized subproblems are still solved) makes this bound even worse. Finally, we  ignore problem-independent parameters $\gamma_1, \gamma_2, \gamma_3$, and $ \eta_1$.} & cubic regularized & $O(\Delta_f  L^{3/2} \sigma_{\min}^{-1} \epsilon^{-3/2} + \Delta_f \sigma_{\min}^{1/2} \epsilon^{-3/2})$ \\
            \textbf{Nesterov et al. \cite[Eq. 5.13 and 5.14]{grapiglia2017regularized}} \tablefootnote{Since by our assumption the function $f$ has L-Lipschitz Hessian, we only consider the case when the H\"{o}lder exponent $\nu = 1$. Note also that the algorithm description \cite[Eq. 5.12]{grapiglia2017regularized}, requires that the initial regularization parameter $\sigma_0$ ($H_0$ using their notation) satisfies $H_0 \in (0, H_f(v)]$ where $H_f(v)$ is defined in \cite[Eq. 2.1]{grapiglia2017regularized}. Technically this condition is not verifiable as $H_f(v)$ is unknown in practice. However, one can readily modify \cite{grapiglia2017regularized} by redefining $H_f(v)$ to be the maximum of $H_0$ and the RHS of \cite[Eq. 2.1]{grapiglia2017regularized} to remove the requirement that $H_0 \in (0, H_f(v)]$. This gives the bound stated in Table~\ref{table:adaptive-second-order-methods}.} & cubic regularized & $O(\Delta_f  \max\{L, \sigma_0 \}^{1/2} \epsilon ^ {-3/2}) + \tilde{O}(1)$ \\
            \textbf{ARp \cite[Section 4.1]{cartis2019universal}} \tablefootnote{We only consider the case for the cubic regularized model when $p = 2$ and $r = p + 1 = 3$. Also, since by our assumption the function $f$ has L-Lipschitz Hessian, we only consider the case when the H\"{o}lder exponent $\beta_2 = 1$.} & cubic regularized &  $O(\Delta_f \max\{L, \sigma_0 \}^{1/2} \epsilon^{-3/2}) + \tilde{O}(1)$\\
            \textbf{TRACE \cite[Section 3.2]{curtis2017trust}} & trust-region & $O(\Delta_f L^{3/2} \epsilon^{-3/2}) + \tilde{O}(1)$\\
            \textbf{Toint et al. \cite[Section 2.2]{curtis2021trust}} & trust-region & $\tilde{O}\left( \Delta_f \max{\left\{ L^2, 1 + 2 L \right\}} \epsilon^{-3/2} \right)$\\
            \textbf{Our method} & trust-region & $O( \Delta_f L^{1/2}  \epsilon^{-3/2}) + \tilde{O}(1)$
            \end{tabular}
            }
\end{table}
}
 
 \textbf{Our contributions:}
 \begin{enumerate}
\item  We present the first consistently adaptive trust-region  method for finding stationary points of nonconvex functions with $L$-Lipschitz Hessians and bounded optimality gap. In particular, we prove our method finds an $\epsilon$-approximate stationary point after at most $O( \Delta_f L^{1/2}  \epsilon^{-3/2}) + \tilde{O}(1)$ iterations.
\item We show our trust-region method has quadratic convergence when 
it enters a region around a point satisfying the second-order sufficient conditions for local optimality.
\item Our method appears promising in experiments. We test our method on the CUTEst test set \cite{gould2015cutest} against other methods including ARC and a classic trust-region  method. These tests show how competitive we are against the other methods in term of total number of required iterations until convergence.  
 \end{enumerate}
 
 \paragraph{Paper outline} The paper is structured as follows.
 Section~\ref{sec:our-tr-method} presents our trust-region
 method and contrasts it with existing trust-region methods. Section~\ref{sec:main-result} presents our main result: a convergence bound for finding
 $\epsilon$-approximate stationary points that is 
 consistently adaptive to problems
 with Lipschitz continuous Hessian. Section~\ref{sec:superlinear-results}
 shows quadratic convergence of the method.
 Section~\ref{sec:experimental-results} discusses the experimental results.
 
 \paragraph{Notation} Let 
 $\N$ be the set of natural numbers (starting from one), $\eye$ be the identity matrix,
 and $\R$ the set of real numbers.
 Throughout this paper
 we assume that $n \in \N$ and $f : \R^{n} \rightarrow \R$ is bounded below and twice-differentiable.
 We define $f_\star := \inf_{x \in \R^{n}} f(x)$ and $\Delta_f := f(x_1) - f_\star$.

\section{Our trust-region method}\label{sec:our-tr-method}
\newcommand{\Mk}{M_k} 

\subsection{Trust-region subproblems}

As is standard for trust-region methods \cite{sorensen1982newton} at each iteration $k$ of our algorithm we
build a second-order Taylor series approximation at
the current iterate $x_k$:
\begin{flalign}
    \Mk(d) := \frac{1}{2} d^T \grad^2 f(x_k) d + \grad f(x_k)^T d 
    \label{eq:second-order-model}
\end{flalign}
and minimize that approximation over a ball with radius $r_k > 0$:
\begin{flalign}\label{eq:trust-region-subproblem}
\min_{d \in \R^{n}} \Mk(d) \text{ s.t. } \| d \| \le r_k
\end{flalign}
to generate a search direction $d_k$.
One important practical question is given a candidate search direction $d_k$, how can we verify that it solves 
\eqref{eq:trust-region-subproblem}.
For this one can use the following well-known Fact.
\begin{fact}[Theorem 4.1 \cite{nocedal2006numerical}] \label{fact:optimality-conditions-quadratic-model}
The direction $d_k$ exactly solves \eqref{eq:trust-region-subproblem} if and only
there exists $\delta_k \in [0,\infty)$ such that:
\begin{subequations}\label{eq:optimality-conditions-tr}
\begin{flalign}
\label{eq:model-gradient1}\grad \Mk(d_k) + \delta_k d_k &= 0 \\
\delta_k r_k &\le \delta_k \| d_k \|  \\
\| d_k \| &\le r_k \\
\grad^2 f(x_k) + \delta_k \eye &\succeq 0 
\label{eq:PSD-grad-squared-f-delta}
\end{flalign}
\end{subequations}
which solves \eqref{eq:trust-region-subproblem}.
\end{fact}

In practice, it is not possible to exactly solve the trust-region subproblem defined in \eqref{eq:trust-region-subproblem}, instead we only require that the trust-region subproblem is approximately solved. For our method, it will suffice to find a direction $d_k$ satisfying:
\begin{subequations}
\begin{flalign}
\label{eq:model-gradient-weaker} \| \grad \Mk(d_k) + \delta_k d_k \| &\le \SPTone \| \grad f(x_k + d_k) \| \\
\label{eq:comp-delta-radius1} \SPTtwo \delta_k r_k  &\le \delta_k \| d_k \| \\
\| d_k \| &\le r_k  \\
\label{eq:model-reduction-by-dist1} \Mk(d_k) &\le -\SPTthree \frac{\delta_k}{2} \| d_k \|^2
\end{flalign}
\label{eq:subproblem-termination-criteria}
\end{subequations}
\hspace{-0.2cm} where $\delta_k$ denotes the solution for the above system and \RequireSPT.
Setting $\SPTone = 0, \SPTtwo = 1, \SPTthree = 1$ represents
the exact version of these conditions.
As Lemma~\ref{lem:exact-tr-sol} shows, exactly solving the trust-region subproblem gives a solution to the system \eqref{eq:subproblem-termination-criteria}.
However, the converse it not true, an exact solution to \eqref{eq:subproblem-termination-criteria} does not necessarily solve the trust-region subproblem.
Nonetheless, these conditions are all we need
to prove our results, and are easier to
verify than a relaxation of \eqref{eq:trust-region-subproblem} that includes a requirement like 
 \eqref{eq:PSD-grad-squared-f-delta} which needs
a computationally expensive eigenvalue calculation.

\begin{lemma}\label{lem:exact-tr-sol}
Any solution to \eqref{eq:optimality-conditions-tr} is a solution to \eqref{eq:subproblem-termination-criteria} with $\SPTone = 0, \SPTtwo = 1, \SPTthree = 1$.
\end{lemma}

\begin{proof}
The only tricky part is proving \eqref{eq:model-reduction-by-dist1}. However, this can be shown using standard arguments:
$M_k(d_k) = \frac{1}{2} d_k^T \grad^2 f(x_k) d_k + \grad f(x_k)^T d_k = -\frac{1}{2} d_k^T (\grad^2 f(x_k) + 2 \delta_k \eye) d_k \le -\frac{\delta_k}{2} \| d_k \|^2$
where the second equality uses \eqref{eq:model-gradient1} and the inequality \eqref{eq:PSD-grad-squared-f-delta}. 
\end{proof}

\subsection{Our trust-region method}\label{sec:our-trust-region-method}

An important component of a trust-region method is the decision for computing the radius $r_k$ at each iteration. This choice is based on whether the observed function value reduction $f(x_k) - f(x_k + d_k)$ is comparable to the predicted reduction from the second-order Taylor series expansion $\Mk$. In particular, given a search direction $d_k$ existing trust-region methods compute the ratio
\begin{flalign}\label{eq:classic-rho-k}
\rho_k := \frac{f(x_k) - f(x_k + d_k)}{-\Mk(d_k)}
\end{flalign}
and then increase $r_k$ if $\rho_k \ge \beta$ or decrease $r_k$ if $\rho_k < \beta$ \cite{sorensen1982newton}. Unfortunately, while intuitive,
this criteria is provably bad, in the sense that one can construct examples of functions with Lipschitz continuous Hessians where any trust-region method that uses this criteria will have a convergence rate proportional to $\epsilon^{-2}$ \cite[Section~3]{cartis2010complexity}.

Instead of (\ref{eq:classic-rho-k}), we introduce a variant of this ratio by adding the term $\frac{\theta}{2} \| \grad f(x_k + d_k) \| \| d_k \|$ to the predicted reduction where
$\theta \in (0,\infty)$ is a problem-independent hyperparameter (we use $\theta = 0.1$ in our implementation). 
This requires the algorithm to reduce the function value more if the gradient norm at the candidate solution $x_{k} + d_k$, and search direction norm are big.
In particular, we define our new ratio as: 
\begin{flalign}\label{eq:compute-rho-k}
\hat{\rho}_k := \frac{f(x_k) - f(x_k + d_k)}{-\Mk(d_k) + \frac{\theta}{2} \| \grad f(x_k + d_k) \| \| d_k \| } 
\end{flalign}
Our trust-region method is presented in Algorithm~\ref{alg:SOAT}. The algorithm includes some other minor modification of classic trust-region
methods \cite{sorensen1982newton}: we accept all search directions that
reduce the function value, and update the $r_{k+1}$ using $\| d_k \|$ instead of $r_k$ \edit{(see \cite[Equation~13.6.13]{sun2006optimization} for a similar update rule)}.
We recommend contrasting our algorithm with \cite[Algorithm 1]{curtis2017trust} which is trust-region method with an iteration bound proportional to $\epsilon^{-3/2}$ but
is more complex and not consistently adaptive.

For the remainder of this paper
$x_k$ and $d_k$ refer to the iterates of Algorithm~\ref{alg:SOAT}.

\begin{algorithm}[tbh]
\caption{\textbf{C}onsistently  \textbf{A}daptive \textbf{T}rust Region Method (CAT)}\label{alg:SOAT}
\textbf{Input requirements:} $r_{1} \in (0,\infty)$, $x_1 \in \R^{n}$ \;
\textbf{Problem-independent parameter requirements:} $\theta \in (0,1), \beta \in (0,1)$, $\omega \in (1,\infty)$, \RequireSPT, $\frac{\beta \theta}{\SPTthree(1 - \beta)} + \SPTone < 1$ \;
\For{$k = 1, \dots, \infty$}{
Approximately solve the trust-region subproblem, i.e., obtain $d_k$ that satisfies \eqref{eq:subproblem-termination-criteria}  \; 
 $x_{k + 1} \gets 
    \begin{cases}
    x_k + d_k & \text{ } f(x_k + d_k) \le f(x_k) \\
    x_k & \text{ otherwise}
    \end{cases}$
    
    $r_{k + 1} \gets 
    \begin{cases}
        \omega \| d_k \| & \text{ } \hat{\rho}_k \ge \beta \\
        \| d_k \| / \omega & \text{ otherwise}
    \end{cases}$
}
\end{algorithm}

\section{Proof of full adaptivity on Lipschitz continuous functions}\label{sec:main-result}

This section proves that our method is consistently adaptive for finding approximate stationary points on functions with $L$-Lipschitz Hessians. 
The core idea behind our proof
is to get a handle on the size of $\| d_k \|$.
In particular, if we can bound $\| d_k \|$ from below
and $\hat{\rho}_k \ge \beta$ then the $\frac{\theta}{2} \| \grad f(x_k + d_k) \| \| d_k \|$
term guarantees that at iteration $k$ the function value
is reduced by a large amount relative to the gradient norm $\| \grad f(x_k + d_k) \|$.

Lemma~\ref{lem:gradient-bound_distance} 
guarantees the norm of the gradient for the candidate solution $x_k + d_k$ lower bounds
the size of $\| d_k \|$ under certain conditions.
Note this bound on the gradient, i.e., \eqref{eq:grad-bound-distance1} holds without us needing to know the Lipschitz constant of the Hessian $L$. The proof of Lemma~\ref{lem:gradient-bound_distance} appears in Section~\ref{sec:lem:gradient-bound_distance} and heavily leverages Fact~\ref{fact:TR-method-accuracy}.

\begin{fact}[Nesterov \& Polyak 2006, Lemma 1 \cite{nesterov2006cubic}]\label{fact:TR-method-accuracy}
If $\grad^2 f$ is $L$-Lipschitz,
\begin{flalign}
\| \grad f(x_k+d_k) \| &\le \| \grad \Mk(d_k) \| + \frac{L}{2} \| d_k \|^2  \label{eq:Lip-grad1}\\
f(x_k + d_k) &\le f(x_k) + \Mk(d_k) + \frac{L}{6} \| d_k \|^3. \label{eq:Lip-progress1}
\end{flalign}
\end{fact}

\begin{lemma} \label{lem:gradient-bound_distance}
Suppose $\grad^2 f$ is $L$-Lipschitz. If
$\| d_k \| < \SPTtwo r_k$ or $\hat{\rho}_k \le \beta$ then
\begin{flalign}\label{eq:grad-bound-distance1}
\| \grad f(x_k + d_k) \| \le c_1 L \| d_k \|^2
\end{flalign}
where $c_1 > 0$ is a problem-independent constant:
$$
c_1 := \max\left\{ \frac{5 - 3 \beta}{6 (\SPTthree(1 - \SPTone)(1 - \beta) - \beta \theta)}, \frac{1}{2 (1 - \SPTone)} \right\}.
$$
\end{lemma}

For the remainder of this section we will find the following quantities useful,
\begin{flalign*}
\dmin &:= \SPTtwo \omega^{-1} c_1^{-1/2} L^{-1/2} \epsilon^{1/2} \\
\dmax &:= \frac{2 \omega}{\beta \theta} \cdot \frac{\Delta_f}{ \epsilon}.
\end{flalign*}
As we will show shortly in Lemma~\ref{lem:bound-direction-sizes}, after a short warm up period $\dmin$ and $\dmax$ represent lower and upper bound on $\|d_k\|$ (i.e., $\dmin \le \| d_k \| \le \dmax$) as long as $\| \grad f(x_k + d_k) \| \ge \epsilon$.
But before presenting and proving Lemma~\ref{lem:bound-direction-sizes} we develop  Lemma~\ref{lem:directions-increase-decrease} which is a stepping stone to proving Lemma~\ref{lem:bound-direction-sizes}.
Lemma~\ref{lem:directions-increase-decrease}
shows that if $\| d_k \|$ is almost above $\dmax$ then the trust-region radius will shrink, and if $\| d_k \|$ is almost below $\dmin$ then the trust-region radius will grow (recall from Algorithm~\ref{alg:SOAT} that $\omega \in (1,\infty)$).

\begin{lemma}\label{lem:directions-increase-decrease}
Suppose $\grad^2 f$ is $L$-Lipschitz.
Let $\epsilon \in (0,\infty)$ and $\| \grad f(x_{k} + d_k) \| \ge \epsilon$ then
\begin{enumerate}
\item If $\| d_k \| > \dmax / \omega $ then $\| d_k \| / \omega = r_{k+1}$. \label{lem:directions-increase-decrease:part-1}

\item  If $\| d_k \| < \omega \SPTtwo ^ {-1} \dmin$ then $\SPTtwo r_k \leq \| d_k \| \leq r_k \And \omega \| d_k \| = r_{k+1}$.\label{lem:directions-increase-decrease:part-2}
\end{enumerate}
\end{lemma}

\begin{proof}
Proof for \ref{lem:directions-increase-decrease:part-1}.
We have
\begin{flalign}\label{eq:when-d-k-is-big}
&\| d_k \| > \dmax / \omega = 2\frac{\Delta_f}{\beta \theta\epsilon} \ge 2\frac{f(x_k)-f(x_{k+1})}{\beta \theta\epsilon}
\end{flalign}

where the first equality uses the definition of $\dmax$ and the second inequality uses $f(x_{k+1}) \ge \inf_{x \in \R^{n}} f(x)$ and $f(x_1) \ge f(x_k)$. Furthermore,
\begin{flalign*}
\hat{\rho}_k &= \frac{f(x_k) - f(x_k + d_k)}{-\Mk(d_k) + \frac{\theta}{2} \| \grad f(x_k + d_k) \| \| d_k \|} \le \frac{f(x_k) - f(x_{k+1})}{\frac{\theta}{2} \| \grad f(x_k + d_k) \| \| d_k \|} \le 2\frac{f(x_k) - f(x_{k+1})}{\theta \epsilon \| d_k \|} < \beta
\end{flalign*}
where the first inequality follows from $-\Mk(d_k) \ge 0$ and $ f(x_{k+1}) \le f(x_k + d_k)$, the second inequality follows from the fact that $\|\grad f(x_k + d_k)\| \ge \epsilon$, and the third inequality uses \eqref{eq:when-d-k-is-big}.
By inspection of Algorithm \ref{alg:SOAT}, if $\hat{\rho}_k < \beta$, then $\| d_k \| / \omega = r_{k+1}$. 

Proof for \ref{lem:directions-increase-decrease:part-2}. 
We will prove the result by contrapositive. 
In particular, suppose that $\neg (\SPTtwo r_k \le \| d_k \| \le r_k)$ \text{ or } $\| d_k \| \omega \neq r_{k+1}$. 
Let us consider these two cases.
If $\neg (\SPTtwo r_k \le \| d_k \| \le r_k)$ then as $\| d_k \| \le r_k$ we have $\SPTtwo r_k > \| d_k \|$.
If $\| d_k \| \omega \neq r_{k+1}$ then by inspection of Algorithm~\ref{alg:SOAT} we have $\hat{\rho}_k \le \beta$. Therefore, in both these cases the premise of Lemma~\ref{lem:gradient-bound_distance} holds.
Now, by $\| \grad f(x_k + d_k) \| \ge \epsilon$ and Lemma~\ref{lem:gradient-bound_distance} we get
$$
\epsilon \le \| \grad f(x_k + d_k) \| \le c_1 L \| d_k \|^2
$$
which implies $\| d_k \| \ge c_1^{-1/2} L^{-1/2} \epsilon^{1/2} = \SPTtwo^{-1} \omega (\SPTtwo \omega^{-1} c_1^{-1/2} L^{-1/2} \epsilon^{1/2}) =  \SPTtwo^{-1} \omega \dmin$.
\end{proof}

\noindent

Now we show that the norm of the direction $d_k$, after some finite iteration $\ubar{k}_\epsilon$, will be  bounded below and above by $\dmin$ and $\dmax$ respectively. For that we first define:
\[
K_\epsilon := \min \{ \{ k \in \N : \| \grad f(x_k + d_k) \| \le \epsilon \} \cup \{\infty\} \}
\]
as the first iteration for which $\| \grad f(x_k + d_k) \| \le \epsilon$, and we also define:
\[
\ubar{k}_\epsilon := \min \{ \{ k \in \N : \dmin \le \| d_k \| \le \dmax \} \cup \{ K_\epsilon - 1 \} \}
\]
as the first iteration for which $\dmin \le \|d_k\| \le \dmax$.
An illustration of Lemma~\ref{lem:bound-direction-sizes}
is given in Figure~\ref{graph:direction-trend}.
In particular, after a certain warm up period the direction norms
can no longer rise above $\dmax$ or below $\dmin$.
Broadly speaking, the idea behind the proof is that
if $\| d_k \|$ is above $\dmax / \omega$ then at the next iteration $\| d_k \|$ decreases and conversely if $\| d_k \|$ is bellow  $\dmin \omega \SPTtwo^{-1}$ then at the next iteration it must increase.

\begin{lemma} \label{lem:bound-direction-sizes}
Suppose $\grad^2 f$ is $L$-Lipschitz
 and let $\epsilon \in (0,\infty)$.
If $\dmin \le \| d_k \| \le \dmax$ then $\dmin \le \| d_j \| \le \dmax$ for all $j \in [k, K_\epsilon) \cap \N$. Furthermore,
$\ubar{k}_\epsilon \le 1 +  \log_{\SPTtwo \omega}(\max\{ 1, \dmin / r_1, r_1 / \dmax \})$.
\end{lemma}

\begin{proof}
We begin by proving $\dmin \le \| d_k \| \le \dmax$ then $\dmin \le \| d_j \| \le \dmax$ for $j \in [k, K_\epsilon) \cap \N$.
We assume that $k < K_\epsilon$ otherwise our desired conclusion clearly holds. We split this proof into two claims.

Our first claim is that $\| d_k \| \le \dmax$ implies $\| d_{k+1} \| \le \dmax$.
We split $\| d_k \| \le \dmax$ into two subcases. If $\| d_k \| \le \dmax / \omega$, then inspection of Algorithm~\ref{alg:SOAT} shows that $\| d_{k+1} \| \le r_{k+1} \le \| d_k \| \omega \le \dmax$.
If $\dmax / \omega \le \| d_k \| \le \dmax$, then Lemma~\ref{lem:directions-increase-decrease}.\ref{lem:directions-increase-decrease:part-1} implies that $\| d_{k+1} \| \le r_{k+1} \le \| d_{k} \| / \omega \le \dmax$.

Our second claim is that $\dmin \le \| d_k \|$ implies $\dmin \le \| d_{k+1} \|$.
We split $\dmin \le \| d_k \|$ into three subcases.
If $\| d_{k+1} \| < \SPTtwo r_{k+1}$, then the contrapositive of Lemma~\ref{lem:directions-increase-decrease}.\ref{lem:directions-increase-decrease:part-2} implies that $\| d_{k+1} \| \ge \dmin$. 
If $\SPTtwo r_{k+1} \leq \| d_{k+1} \| \leq r_{k+1}$ and $ \dmin \le \| d_{k} \| < \SPTtwo ^ {-1} \omega \dmin$, then $\dmin < \SPTtwo \omega \dmin \le \SPTtwo \omega \| d_{k} \| =_\star \SPTtwo r_{k+1} \leq \| d_{k+1} \|$ where $\star$ uses Lemma~\ref{lem:directions-increase-decrease}.\ref{lem:directions-increase-decrease:part-2}.
If $\SPTtwo r_{k+1} \leq \| d_{k+1} \| \leq r_{k+1}$ and $\dmin \omega \SPTtwo ^ {-1} \le \| d_k \|$, then $\dmin \le \frac{\SPTtwo \|d_k\|}{\omega} \leq_\star \SPTtwo r_{k+1}\leq \| d_{k + 1} \|$ where $\star$ is from
the update rule for $r_{k+1}$ in Algorithm~\ref{alg:SOAT}.

By induction on the previous two claims we deduce if $\dmin \le \| d_k \| \le \dmax$ then $\dmin \le \| d_j \| \le \dmax$ for $j \in [k, K_\epsilon) \cap \N$.

Next, we prove that if
$$
k \ge 1 + \log_{\omega \SPTtwo}( \dmin / r_1 )
$$
then $\| d_k \| \ge \dmin$.
If $\dmin \le \| d_j \|$ for some $j \le k$, then the result holds because as we already established  $\dmin \le \| d_k \|
\Rightarrow \dmin \le \| d_{k+1} \|$.
On the other hand, if $\| d_j \| < \dmin$ for all $j \le k$, then Lemma~\ref{lem:directions-increase-decrease}.\ref{lem:directions-increase-decrease:part-2} implies that $r_{j+1} = \omega \| d_j \| \geq \omega \SPTtwo r_j$ which by induction gives $\| d_{k} \| \geq (\omega \SPTtwo)^{k-1} r_1 \ge (\omega \SPTtwo)^{\log_{\omega \SPTtwo}(\dmin / r_1)} r_1 = \dmin$.

Finally, we prove that if 
$$
k \ge 1 + \log_{\omega}( r_1 / \dmax)
$$
then $\| d_k \| \le \dmax$.
If $\| d_j \| \le \dmax $ for some $j \le k$, then the result holds because as we already established  $ \| d_k \| \le \dmax
\Rightarrow \| d_{k+1} \| \le \dmax$.
On the other hand, if $\| d_j \| > \dmax$ for all $j \le k$, then Lemma~\ref{lem:directions-increase-decrease}.\ref{lem:directions-increase-decrease:part-1} implies that $r_{j+1} = \| d_j \| / \omega \le r_j / \omega $ which by induction gives $\| d_{k} \| \le \omega^{1-k} r_1 \le \omega^{-\log_{\omega}(r_1 / \dmax)} r_1 = \dmax$.
\end{proof}

\begin{figure}[t]
 \centering
\includegraphics[scale=0.33]{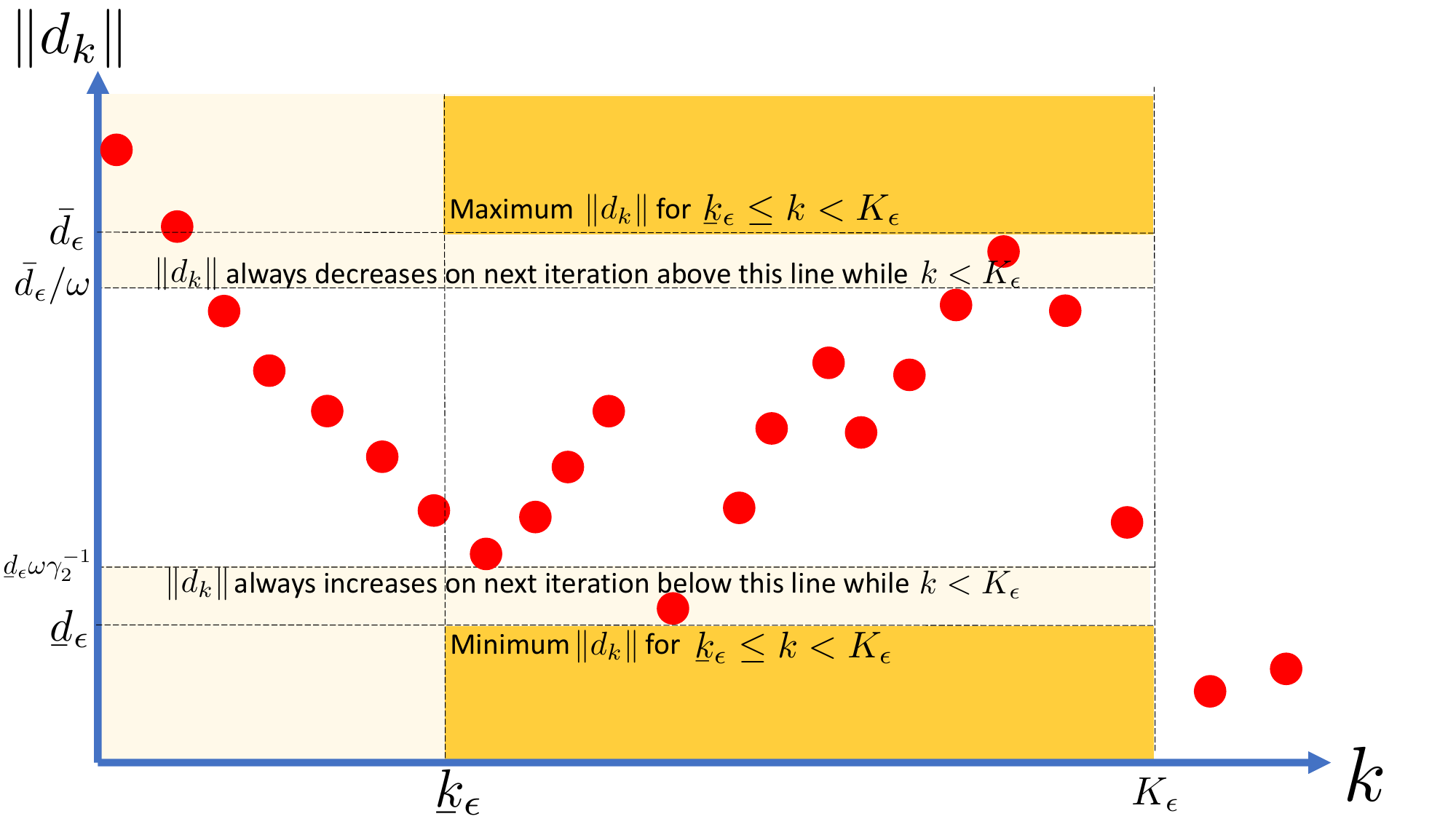}
\caption{
An example of a plausible sequence of iterates and the norms of
their directions. Each red dot represents an iterate and its search direction norm. This illustrates  Lemma~\ref{lem:bound-direction-sizes}.}
\label{graph:direction-trend}
\end{figure} 

Let
$$
\mathcal{P}_\epsilon := \{ k \in \N : \hat{\rho}_k \ge \beta, \ubar{k}_\epsilon \le  k < K_\epsilon \}
$$
which represents the set of iterations, before we find an $\epsilon$-approximate stationary point, where the function value is reduced a large amount compared with our target reduction, i.e., $\hat{\rho}_k \ge \beta$.
Lemma~\ref{lem:increase-radius-index} shows that there is a finite number of these iterations until the gradient drops below the target threshold $\epsilon$.
The proof of Lemma~\ref{lem:increase-radius-index} appears in Appendix~\ref{sec:lem:increase-radius-index}.
Roughly, the idea of the proof is to use
that, due our definition of 
$\hat{\rho}_k$, when $\hat{\rho}_k \ge \beta$ we always reduce the function value by at least
$\frac{\beta \theta}{2} \| \grad f(x_k + d_k) \| \| d_k \|$
and $\| d_k \|$ can be lower bounded
by $\dmin$ using Lemma~\ref{lem:bound-direction-sizes}.
As we cannot reduce the function value
by a constant value indefinitely, we must eventually have $\| \grad f(x_k + d_k) \| \le \epsilon$.

\begin{lemma} \label{lem:increase-radius-index}
Suppose $\grad^2 f$ is $L$-Lipschitz and $\epsilon \in (0,\infty)$ then
$\abs{ \mathcal{P}_\epsilon } \le \frac{\dmax}{\dmin \omega} + 1 = \frac{2 \omega c_1^{1/2}}{\beta \theta} \cdot \frac{\Delta_f L^{1/2}}{\epsilon^{-3/2}} + 1$.
\end{lemma}

With Lemma~\ref{lem:increase-radius-index} in hand we are now ready to prove our main result,
Theorem~\ref{thm:main-fully-adaptive-result}.
We have already provided a bound on the length
of the warm up period, $\ubar{k}_\epsilon$ (Lemma~\ref{lem:bound-direction-sizes}) and on the number of points with $\hat{\rho}_k \ge \beta$.
Therefore, the only obstacle is to bound the number
of points with $\hat{\rho}_k < \beta$. However,
on these iterations we always decrease the radius
by at least $\omega$ (see update rules in Algorithm~\ref{alg:SOAT}), and therefore as $\| d_k \|$ is bounded below by $\dmin$, there must be
iterations where we increase the radius $r_k$, which by definition of Algorithm~\ref{alg:SOAT} only occurs if $\beta \ge \hat{\rho}_k$. Consequently, the number of iterations where $\beta < \hat{\rho}_k$ can be bounded by the number of iterations where $\beta \ge \hat{\rho}_k$ plus a $\tilde{O}(1)$ term.
This is the crux of the proof of Theorem~\ref{thm:main-fully-adaptive-result} which appears in Section~\ref{sec:main-fully-adaptive-result}.

\begin{theorem}\label{thm:main-fully-adaptive-result}
Suppose that $\grad^2 f$ is $L$-Lipschitz and $f$ is bounded below with $\Delta_f = f(x_1) - f_\star$, then for all $\epsilon \in (0,\infty)$ there exists some iteration $k$ with $\| \grad f(x_k + d_k) \| \le \epsilon$ and
$$
k \le O\left( \frac{\Delta_f L^{\frac{1}{2}}}{ \epsilon^{\frac{3}{2}}} + \log\left(  \frac{\epsilon ^ \frac{1}{2}}{L ^ {\frac{1}{2}} r_1} +  \frac{r_1 \epsilon}{\Delta_f} + 1 \right) + 1 \right)
$$
where $O( \cdot )$ hides problem-independent constant factors and $r_1$ is the initial trust-region radius.
\end{theorem}

One drawback of Theorem~\ref{thm:main-fully-adaptive-result} is that it only bounds the number of iterations
to find a first-order stationary point.
Many second-order methods in the literature show convergence 
to points satisfying the second-order optimality conditions \cite{birgin2017use,cartis2011adaptiveII,curtis2017trust,zhang2022adaptive}.
Of course, these methods are not consistently adaptive. Therefore, in the future, it would be
interesting to develop a method
that provides a consistently adaptive 
convergence guarantee for finding
second-order stationary points.

\section{Quadratic convergence when sufficient conditions for local optimality hold}\label{sec:superlinear-results}

\begin{theorem}\label{thm:main-superlinear-convergence-result}
Suppose $f$ is twice differentiable and for some $x_\star \in R^{n}$ the second-order sufficient conditions for local optimality hold ($\grad f(x_\star) = 0$ and
$\grad^2 f(x_\star) \succ 0$). Under these conditions there exists a neighborhood $N$ around $x_\star$ and a constant $c > 0$ such that if $x_i \in \NBH$ then there exist $x_{\hat{k}} \in \NBH$ such that for all $k \ge \hat{k}$ we have
$\| x_{k+1} - x_\star \| \le c \| x_k - x_\star \|^2 \le \frac{1}{2} \| x_k - x_\star \|$.
\end{theorem}

The proof of Theorem~\ref{thm:main-superlinear-convergence-result} appears in Section~\ref{sec:proof-of-superlinear-convergence}.
It is a little tricker than typical quadratic convergence proofs for trust-region
methods because in our method
we have $\lim_{k \rightarrow 0} r_k \rightarrow 0$ whereas classical trust-region methods have $r_k$ bounded away from zero \cite[Proof of Theorem~4.14]{sorensen1982newton}. Fortunately,
one can show that asymptotically $r_k \ge \omega \| x_k - x_\star \|$ so the decaying radius does not interfere with quadratic convergence. 
In particular, the crux of proving
Theorem~\ref{thm:main-superlinear-convergence-result} is proving
the premise of  Lemma~\ref{lem:superlinear-big-r-k} (as the conclusion of Lemma~\ref{lem:superlinear-big-r-k} is $r_k \ge \omega \| x_k - x_\star \|$).
For this Lemma we define $\diam(X) := \sup_{x, x' \in X} \| x - x' \|$.

\begin{lemma}\label{lem:superlinear-big-r-k}
Let $\NBH$ be a bounded set such that for all $x_k \in \NBH$ we have $x_{k+1} \in \NBH$,
$\hat{\rho}_k \ge \beta$, 
and $\min\{ \SPTtwo r_{k}, \| x_{k+1} - x_\star \| \} \le \| d_k \| \le  \omega \SPTtwo \| x_{k} - x_\star \| $.
Suppose there exists $x_k \in \NBH$
and let $i$ be the smallest index with $x_i \in \NBH$,
then $r_k \ge \omega \| x_k - x_\star \|$ for all
$k \ge 2 + i + \log_{\SPTtwo \omega}(\frac{\diam(\NBH)}{\| d_i \|})$.
\end{lemma}

\begin{proof}
Let $k \ge i$. By induction $x_k \in \NBH$. By $\hat{\rho}_k \ge \beta$ and inspection of Algorithm~\ref{alg:SOAT} we have $r_{k+1} = \omega \| d_k \|$.
Suppose that $\| d_k \| \ge \SPTtwo r_{k}$ for all $i \le k \le i + \ell$ then $\diam(\NBH) \ge \| d_{k+1} \| \ge \SPTtwo r_{k+1} = \SPTtwo \omega \| d_k \| = \SPTtwo^{\ell} \omega^{\ell} \| d_i \|$. Rearranging gives $\ell \le \log_{\SPTtwo \omega}(\frac{\diam(\NBH)}{\| d_i \|})$. 
Next observe that if 
$\| d_{k} \| < \SPTtwo r_k$ then $\| d_{k+1} \| \le \omega \SPTtwo \| x_{k+1} - x_\star \| \le \omega \SPTtwo \| d_{k} \| = (\omega \SPTtwo/\omega) r_{k+1} = r_{k+1}$. By induction $\| d_k \| < \SPTtwo r_k$ for all $k > \ell$.
Finally,
observe that if $\| d_k \| < \SPTtwo r_k$ then $r_{k+1} = \omega \| d_k \| \ge \omega \| x_{k+1} - x_\star \|$.
\end{proof}

\section{Experimental results}\label{sec:experimental-results}

We test our algorithm on learning linear dynamical systems \cite{hardt2016gradient}, matrix completion \cite{zhuang2020power}, and the CUTEst test set \cite{orban-siqueira-cutest-2020}. 

Appendix~\ref{sec:experimental-full-results} contains the complete set of results from our experiments. Our method is implemented in an open-source Julia module \edit{available at \url{https://github.com/fadihamad94/CAT-NeurIPS}}. The implementation uses a factorization and eigendecomposition approach to solve the trust-region subproblems (i.e., satisfy \eqref{eq:subproblem-termination-criteria}). 
We perform our experiments using Julia 1.6 on a Linux virtual machine that has 8 CPUs and 16 GB RAM. \edit{The CAT code repository provides instructions for reproducing the experiments and detailed tables of results.}

For these experiments, the selection of the parameters (unless otherwise specified) is as follow: $r_1 = 1.0$, $\beta = 0.1$, $\theta = 0.1$, $\omega = 8.0$, $\SPTone = 0.0$, $\SPTtwo = 0.8$, and $\SPTthree = 1.0$.
When implementing Algorithm \ref{alg:SOAT} with some target tolerance $\epsilon$, we immediately terminate when we observe a point $x_k$ with $\|\grad f(x_k + d_k) \| \leq \epsilon$. This also includes the case when we check the inner termination criteria for the trust-region subproblem. 
The full details of the implementation are described in Appendix~\ref{sec:solving-tr-subproblem}.

\subsection{Learning linear dynamical systems}
We test our method on learning linear dynamical systems \cite{hardt2016gradient} to see how efficient our method compared to a trust-region solver. 
We synthetically generate an example with noise
both in the observations and also the evolution of
the system, and then recover the parameters
using maximum likelihood estimation.
Details are provided in Appendix~\ref{sec:experimental-learning-linear-dynamical-systems}.

On this problem we compare our algorithm with a Newton trust-region method that is available through the Optim.jl package \cite{mogensen2018optim}. 
The comparisons are summarized in Table \ref{table:learning-linear-dynamical-system}.

\begin{table}[tb]
\caption{Geometric mean for total number of iterations, and evaluations of the function and gradient, per solver on 60 randomly generated instances with $\epsilon = 10^{-5}$ termination tolerance. Failures counted as the maximum number of iterations ($10000$) when computing the geometric mean.}
\label{table:learning-linear-dynamical-system}
\centering
\begin{tabular}{llll}
\hline
& \textbf{\#iter} & \textbf{\#function} & \textbf{\#gradient} \\ \midrule
\textbf{Newton trust-region \cite{mogensen2018optim}} & 480.1 & 482.3 & 482.3\\
\textbf{Our method} & 308.1 & 309.6 & 309.6 \\
\textbf{$95 \%$ CI for ratio} & [1.24, 1.95] & [1.24, 1.94] & [1.24, 1.94]
\end{tabular}
\end{table}

\subsection{Matrix completion}
\edit{
We also demonstrate the effectiveness of our algorithm against a trust-region solver on the matrix completion problem. The matrix completion formulation can be written as the regularized squared error function of SVD model \cite[Equation 10]{zhuang2020power}. For our experiment, we use the public data set of Ausgrid, but we only use the data from a single substation.  Details are provided in Appendix \ref{sec:experimental-cmatrix-completion}.

Again we compare our algorithm with a Newton trust-region method \cite{mogensen2018optim}. The comparison are summarized in Table \ref{table:matrix-completion}.

\begin{table}[tb]
\caption{Geometric mean for total number of iterations per solver on 10 instances by randomly
generating the sampled measurements from the matrix D using data from Ausgrid with $\epsilon = 10^{-5}$ termination tolerance. Failures counted as the maximum number of iterations ($1000$) when computing the geometric mean.}
\label{table:matrix-completion}
\centering
\begin{tabular}{ll}
\hline
& \textbf{\#iter} \\ \midrule
\textbf{Newton trust-region \cite{mogensen2018optim}} & 1000\\
\textbf{Our method} & 216.4
\end{tabular}
\end{table}
}

\subsection{Results on CUTEst test set}\label{sec:CUTEst-results}

The CUTEst test set \cite{orban-siqueira-cutest-2020} is a
standard test set for nonlinear optimization algorithms.
\edit{To run the benchmarks we use \url{https://github.com/JuliaSmoothOptimizers/CUTEst.jl} (the License can be found at \url{https://github.com/JuliaSmoothOptimizers/CUTEst.jl/blob/main/LICENSE.md}).}
We will be comparing with the results for ARC
reported in \cite[Table~1]{cartis2011adaptiveI}.
\edit{As the benchmark CUTEst that we used has changed since \cite{cartis2011adaptiveI} was written we select only the problems in CUTEst that remain the same (some of the problem sizes have changed).
This gives 67 instances.
A table with our full results can be found in the results/CUTEst subdirectory in the git repository for this paper. }
Catris et.al \cite{cartis2011adaptiveI} report the results for three different implementations of their ARC algorithm. We limit our comparison to the ARC g-rule algorithm since it performs better than the other ARC approaches. We also run the Newton trust-region method from the Optim.jl package \cite{mogensen2018optim}.

Our algorithm is stopped as soon $\|\grad f(x_k + d_k)\|$ is smaller than $10 ^ {-5}$. 
For the Newton trust-region method \cite{mogensen2018optim} we also used as a stopping criteria a value of $10 ^ {-5}$ for the gradient termination tolerance. We used $10000$ as an iteration limit and any run exceeding this is considered a failure. This choice of parameters is to be consistent with \cite{cartis2011adaptiveI}. 

As we can see from Table \ref{table:cutest-run-statistics}, our algorithm offers significant promise, requiring similar function evaluations (and therefore subproblem solves) to converge than the Newton trust-region of \cite{mogensen2018optim} and ARC, although the number of gradient evaluations is slightly higher than ARC. In addition, the comparison between these algorithms in term of total number of iterations and total number of gradient evaluations is summarized in Figure \ref{fig:performance-fraction-problems-solved}.

\begin{table}[tb]
\caption{Number of failures, geometric mean for total number of iterations and function and gradient evaluations per solver on 67 unconstrained instance from the CUTEst benchmark instances. Failures counted as the maximum number of iterations ($10000$) when computing the geometric mean.}
\label{table:cutest-run-statistics}
\centering
\begin{tabular}{lcccc}
\hline
\textbf{}                                               & \textbf{\#failures} & \textbf{\#iter} & \textbf{\#function} & \textbf{\#gradient} \\ \midrule
\textbf{Our method} & 3 & 41.5 & 44.4 & 44.4 \\
\textbf{ARC with the g-rule \cite{cartis2011adaptiveI}} & 1& 38.1 & 38.1 & 26.6 \\
\textbf{Newton trust-region \cite{mogensen2018optim}}  & 4  & 44.5 & 47.2 & 47.2               
\end{tabular}
\end{table}

\begin{figure}[tb]
\centering
\includegraphics[scale=0.3]{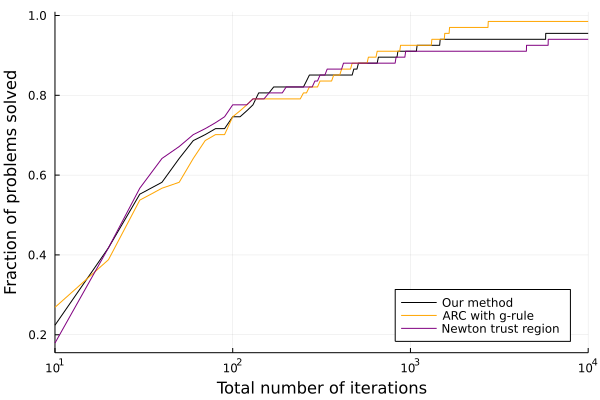}
\includegraphics[scale=0.3]{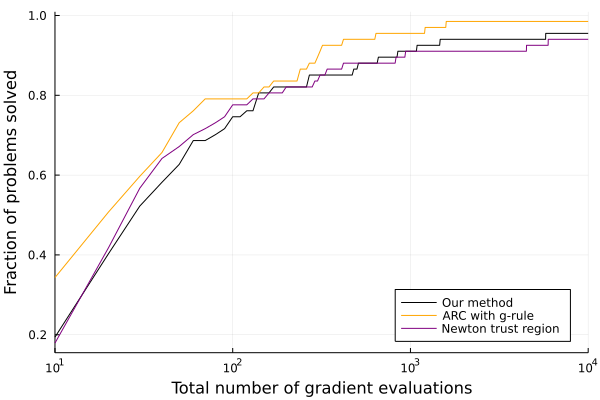}
\caption{Fraction of problems solved  on the CUTEst benchmark.
}
\label{fig:performance-fraction-problems-solved}
\end{figure} 

\subsection{Convergence of our method with different theta values}

To demonstrate the role of the additional $\frac{\theta}{2} \| d_k \| \| \grad f(x_k + d_k ) \|$ term
we run experiments with different $\theta$ values. In
particular, we contrast our default value of $\theta=0.1$
with $\theta=0$ which corresponds to not adding the
$\frac{\theta}{2} \| d_k \| \| \grad f(x_k + d_k ) \|$ term to $\hat{\rho}_k$ (recall the discussion in Section~\ref{sec:our-trust-region-method}).

In Figure~\ref{fig:different-theta:CUTESt} we rerun on the CUTEst test set (as per Section~\ref{sec:CUTEst-results}) and
compare these two options. One can see the algorithm
performs similarly with either $\theta = 0$ or
$\theta = 0.1$.

In Figure~\ref{fig:different-theta:hard-example} we test on the hard example from
\cite{cartis2010complexity} which is designed
to exhibit the poor worst-case complexity of 
trust-region methods (i.e., a convergence
rate proportional to $\epsilon^{-2}$) if
the initial radius $r_1$ is chosen sufficiently large
(to achieve this we set $r_1 = 1.5$).
We run this example with $\epsilon=10^{-3}$.
One can see that while for the first $\approx10^4$
iterations the methods follow identical trajectories,
thereafter $\theta=0.1$ rapidly finds a stationary
point whereas $\theta=0.0$ requires two orders of
magnitude more iterations to terminate.
This crystallizes the importance of $\theta$ in circumventing
the worst-case $\epsilon^{-2}$ convergence rate of trust-region methods.

\begin{figure}[tb]
\begin{center}
\begin{subfigure}{0.45\textwidth}
      \centering
      \includegraphics[scale=0.3]{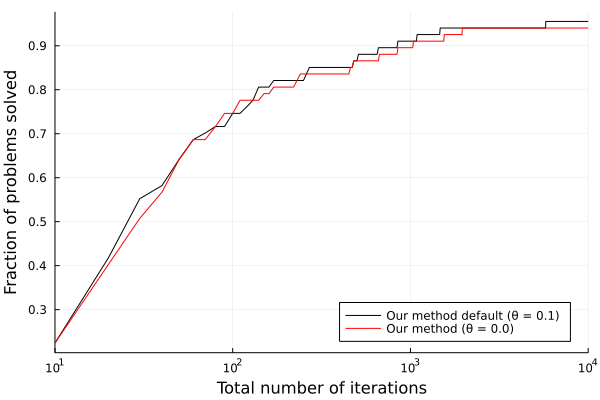}
      \caption{Fraction of problems solved versus total number of iterations on the CUTEst test set}\label{fig:different-theta:CUTESt}
    \end{subfigure}
    \begin{subfigure}{0.45\textwidth}
      \centering
    \includegraphics[scale=0.3]{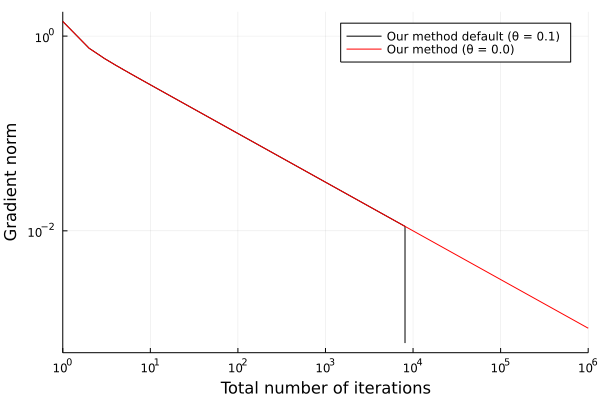}
      \caption{Gradient norm versus total number of iterations on the example of \cite{cartis2010complexity} with $r_1 = 1.5$}\label{fig:different-theta:hard-example}
    \end{subfigure}
\end{center}
\caption{Performance of our algorithm based on different $\theta$ values.}
\label{fig:different-theta}
\end{figure}

\begin{ack}
The authors were supported by the Pitt Momentum Funds.
The authors would like to thank Coralia Cartis and Nicholas Gould for their
helpful feedback on a draft of this paper.
The authors would also like to thank Xiaoyi Qu for finding errors in the proof of an early draft
of the paper and for helpful discussions.
\end{ack}

\newpage

\bibliographystyle{unsrt}
\bibliography{references.bib}

\section*{Checklist}

\begin{enumerate}

\item For all authors...
\begin{enumerate}
  \item Do the main claims made in the abstract and introduction accurately reflect the paper's contributions and scope?
    \answerYes{}
  \item Did you describe the limitations of your work?
    \answerYes{The comparisons with other algorithms show the limitations of our work.}
  \item Did you discuss any potential negative societal impacts of your work?
    \answerNA{}
  \item Have you read the ethics review guidelines and ensured that your paper conforms to them?
    \answerYes{}{}
\end{enumerate}

\item If you are including theoretical results...
\begin{enumerate}
  \item Did you state the full set of assumptions of all theoretical results?
    \answerYes{See Sections \ref{sec:main-result} and \ref{sec:superlinear-results}}
        \item Did you include complete proofs of all theoretical results?
    \answerYes{See Sections \ref{sec:main-result} and  \ref{sec:superlinear-results}}
\end{enumerate}

\item If you ran experiments...
\begin{enumerate}
  \item Did you include the code, data, and instructions needed to reproduce the main experimental results (either in the supplemental material or as a URL)?
    \answerYes{Code and instructions to reproduce the results are attached in the supplementary materials.}
  \item Did you specify all the training details (e.g., data splits, hyperparameters, how they were chosen)?
    \answerNA{}
        \item Did you report error bars (e.g., with respect to the random seed after running experiments multiple times)?
    \answerYes{}
        \item Did you include the total amount of compute and the type of resources used (e.g., type of GPUs, internal cluster, or cloud provider)?
    \answerNA{}
\end{enumerate}

\item If you are using existing assets (e.g., code, data, models) or curating/releasing new assets...
\begin{enumerate}
  \item If your work uses existing assets, did you cite the creators?
    \answerYes{We use CUTEst.jl and Optim.jl. For both we cite the creators.}
  \item Did you mention the license of the assets?
    \answerYes{See Section~\ref{sec:CUTEst-results} and Appendix~\ref{sec:experimental-learning-linear-dynamical-systems}.}
  \item Did you include any new assets either in the supplemental material or as a URL?
    \answerYes{Our code is attached in the supplementary materials.}
  \item Did you discuss whether and how consent was obtained from people whose data you're using/curating?
    \answerNo{}{}
  \item Did you discuss whether the data you are using/curating contains personally identifiable information or offensive content?
    \answerNA{}
\end{enumerate}

\item If you used crowdsourcing or conducted research with human subjects...
\begin{enumerate}
  \item Did you include the full text of instructions given to participants and screenshots, if applicable?
    \answerNA{}
  \item Did you describe any potential participant risks, with links to Institutional Review Board (IRB) approvals, if applicable?
    \answerNA{}
  \item Did you include the estimated hourly wage paid to participants and the total amount spent on participant compensation?
    \answerNA{}
\end{enumerate}

\end{enumerate}


\newpage
\appendix

\section{Proof of results from Section~\ref{sec:main-result}}

\subsection{Proof of Lemma~\ref{lem:gradient-bound_distance}}\label{sec:lem:gradient-bound_distance}

\begin{proof}
First we prove result in the case that $\| d_k \| < \SPTtwo r_k$. By \eqref{eq:comp-delta-radius1} the statement $\| d_k \| < \SPTtwo r_k$ implies $\delta_k = 0$. Combining $\delta_k = 0$ with \eqref{eq:model-gradient-weaker} and \eqref{eq:Lip-grad1} and using the fact $1 - \SPTone > 0$ yields
$$
\| \grad f(x_k + d_k) \| \le \frac{L}{2(1-\SPTone)}  \| d_k \|^2 \leq  c_1 L \| d_k \|^2 \ .
$$

Next we prove the result in the case that $ \hat{\rho}_k \le \beta$. Then
\begin{flalign*}
\Mk(d_k) + \frac{L}{6} \| d_k \|^3 &\ge  f(x_k + d_k) - f(x_k) = -\hat{\rho}_k  \left( -\Mk(d_k) + \frac{\theta}{2} \| \grad f(x_k + d_k) \| \| d_k \| \right) \\
&\ge -\beta \left( -\Mk(d_k) + \frac{\theta}{2} \| \grad f(x_k + d_k) \| \| d_k \|\right)
\end{flalign*}
where the the first inequality uses \eqref{eq:Lip-progress1}, the first equality uses the definition of $\hat{\rho}_k$, and the second inequality uses $\hat{\rho}_k \le \beta$ and $-\Mk(d_k) + \frac{\theta}{2} \| \grad f(x_k + d_k) \| \| d_k \| \ge 0$.

Rearranging the previous inequality using $1 - \beta > 0$ and then applying \eqref{eq:model-reduction-by-dist1} yields:
\begin{flalign}\label{eq:lemma13:0328j}
\frac{L}{3 (1 - \beta )} \| d_k \|^2 + \frac{\beta \theta}{1 - \beta } \| \grad f(x_k + d_k) \|
\ge -\frac{2 \Mk(d_k)}{\| d_k \|} \ge \SPTthree \delta_k \| d_k \|.
\end{flalign}
Now, by \eqref{eq:Lip-grad1},  \eqref{eq:model-gradient-weaker} and the triangle inequality, and \eqref{eq:lemma13:0328j} respectively:
 \begin{flalign*}
\| \grad f(x_k + d_k) \| &\le \| \grad \Mk(d_k) \| + \frac{L}{2} \| d_k \|^2  \le \delta_k \| d_k \| + \SPTone \| \grad f(x_k + d_k) \| + \frac{L}{2} \| d_k \|^2 \\
&\le L \left( \frac{1}{3 \SPTthree(1 - \beta)} + \frac{1}{2} \right) \| d_k \|^2 + \left( \frac{\beta \theta}{\SPTthree(1 - \beta)} + \SPTone \right) \| \grad f(x_k + d_k) \|.
\end{flalign*} Rearranging the latter inequality for $\| \grad f(x_k + d_k) \|$ and using $\frac{\beta \theta}{\SPTthree(1 - \beta)} + \SPTone < 1$ from the requirements of Algorithm~\ref{alg:SOAT} yields:
\begin{flalign*}
\| \grad f(x_k + d_k) \|
&\le  \frac{\frac{1}{3 \SPTthree (1 - \beta)} + \frac{1}{2}}{1 - \frac{\beta \theta}{\SPTthree(1 - \beta)} - \SPTone} L \| d_k \|^2 = \frac{2 + 3 \SPTthree (1 - \beta)}{6 (\SPTthree(1 - \SPTone) (1 - \beta) - \beta \theta)} L\| d_k \|^2 \\
&\le \frac{5 - 3\beta}{6 (\SPTthree(1 - \SPTone) (1 - \beta) - \beta \theta)} L\| d_k \|^2.
\end{flalign*}
\end{proof}

\subsection{Proof of Lemma~\ref{lem:increase-radius-index}}\label{sec:lem:increase-radius-index}

\begin{proof}
For conciseness let $m = \abs{ \mathcal{P}_\epsilon }$.
Suppose that the indices of $\mathcal{P}_\epsilon$ are ordered increasing value by a permutation function $\pi$, i.e., $\mathcal{P}_\epsilon = \{ \pi(i) : i \in [m] \}$ with $\pi(1) < \dots < \pi(m)$.
Then
\edit{
\begin{flalign*}
\Delta_f &\ge f(x_{\pi(1)}) - f(x_{\pi(m)}) = \sum_{i=1}^{m-1} f(x_{\pi(i)}) - f(x_{\pi(i+1)}) 
\end{flalign*}
where the first inequality uses the fact that $f(x_{\pi(i)})$ is non-increasing in $\pi(i)$ and $f(x_{\pi(i)}) \ge f_\star$ and the equality is simply the definition of the telescoping sum of $f(x_{\pi(m)}) - f(x_{\pi(1)})$. Therefore,

\begin{flalign*}
\Delta_f &\ge \sum_{i=1}^{m-1} f(x_{\pi(i)}) - f(x_{\pi(i+1)}) = \sum_{i=1}^{m-1} \hat{\rho}_{\pi(i)} \left( -\Mk({d_{\pi(i)}}) + \frac{\theta}{2} \| \grad f(x_{\pi(i)} +d_{\pi(i)}) \| \| d_{\pi(i)} \| \right) \\
&\ge \sum_{i=1}^{m-1} \beta \left( -\Mk(d_{\pi(i)}) + \frac{\theta}{2} \| \grad f(x_{\pi(i)} + d_{\pi(i)}) \| \| d_{\pi(i)} \| \right) \ge  \frac{\beta \theta}{2}  \sum_{i=1}^{m-1} \| \grad f(x_{\pi(i)} + d_{\pi(i)}) \| \| d_{\pi(i)} \|  \\
&\ge \frac{\epsilon \beta \theta}{2} (m-1) \dmin
\end{flalign*}
where the first equality uses the definition of $\hat{\rho}_{\pi(i)}$, the second inequality follows from $\hat{\rho}_{\pi(i)} \ge \beta$ for $\pi(i) \in \mathcal{P}_\epsilon$, the third inequality uses that $-\Mk(d_{\pi(i)}) \ge 0$, the final inequality uses that $\pi(i) \in \mathcal{P}_\epsilon$ implies that $\| \grad f(x_{\pi(i)} + d_{\pi(i)}) \| \ge \epsilon$ (by definition of $\pi(i) \in \mathcal{P}_\epsilon$) and $\dmin \le \| d_{\pi(i)} \|$ (due to Lemma~\ref{lem:bound-direction-sizes}).
}

Rearranging the latter inequality for $m$ using the fact that $\beta \theta \epsilon \dmin > 0$ and $\Delta_f \ge 0$ yields $m \le \frac{2 \Delta_f}{\beta \theta \epsilon \dmin} + 1 = \frac{\dmax}{\dmin \omega} + 1 = $ where the equalities use the definitions of $\dmax$ and $\dmin$.
\end{proof}

\subsection{Proof of Theorem~\ref{thm:main-fully-adaptive-result}}\label{sec:main-fully-adaptive-result}

\begin{proof}
Define:
\begin{flalign*}
n_j &:= \abs{\{ k \in \N : k \not\in \mathcal{P}_\epsilon, k < K_\epsilon, \ubar{k}_\epsilon < k \le j \}} \\
p_j &:= \abs{ \{ k \in \mathcal{P}_\epsilon : \ubar{k}_\epsilon < k \le j \}}.
\end{flalign*}
First we establish that
\begin{flalign}\label{eq:bound-n}
n_{\infty} \le p_{\infty} + \log_{\omega}\left( \max\left\{ \frac{\dmax}{\dmin}, 1 \right\} \right).
\end{flalign}
Consider the induction hypothesis that 
\begin{flalign}\label{eq:r-k-bound}
r_{k} \le r_{\ubar{k}_\epsilon} \omega^{p_k - n_k} 
\quad \forall k \in [\ubar{k}_\epsilon, K_\epsilon) \cap \N.
\end{flalign}
If $k = \ubar{k}_\epsilon$ then $p_k = n_k = 0$ and the hypothesis holds.
Suppose that the induction hypothesis holds for $k = j$.
Note that for all $j \in \N$ either $p_{j+1} = p_j + 1$ (and $n_{j+1} = n_j$) or $n_{j+1} = n_j + 1$ (and $p_{j+1} = p_j$).
If $p_{j+1} = p_j + 1$ then
$$
r_{j+1} = \| d_j \| \omega \le r_j \omega \le  r_{\ubar{k}_\epsilon} \omega^{p_j-n_j+1} =  r_{\ubar{k}_\epsilon} \omega^{p_{j+1}-n_{j+1}}.
$$
On the other hand, if $n_{j+1} = n_j + 1$ then
$$
r_{j+1} = \| d_j \| / \omega \le r_j / \omega \le r_{\ubar{k}_\epsilon} \omega^{p_j-n_j-1} = r_{\ubar{k}_\epsilon} \omega^{p_{j+1}-n_{j+1}}.
$$
Therefore by induction \eqref{eq:r-k-bound} holds.
By \eqref{eq:r-k-bound} and Lemma~\ref{lem:bound-direction-sizes},
$$
\dmin \le \dmax \omega^{p_k - n_k}
$$
which establishes \eqref{eq:bound-n}. 

By Lemma~\ref{lem:bound-direction-sizes} we have $\ubar{k}_\epsilon \le 1 + \log_{\SPTtwo \omega}(\max\{ 1,  \dmin / r_1, r_1 / \dmax \})$  and Lemma~\ref{lem:increase-radius-index} we have
$p_{\infty} \le \frac{\dmax}{\dmin \omega} + 1$;
using these inequalities in conjuction with \eqref{eq:bound-n} gives 
\begin{flalign*}
K_\epsilon &= \ubar{k}_\epsilon + p_{\infty} + n_{\infty}+1 \le \ubar{k}_\epsilon + 2 p_{\infty} + \log_{\omega}\left( \max\{ \dmax / \dmin \} \right) +1 \\
&\le  \log_{\omega \SPTtwo}(\max\{ 1, \dmin / r_1, r_1 / \dmax \}) + \frac{2 \dmax}{\dmin \omega} + \log_{\omega}( \max\{1, \dmax / \dmin \} ) + 3 \\
&\le \frac{2 \dmax}{\dmin \omega} + 2 \log_{\omega \SPTtwo}\left( \max\left\{ \frac{\dmax}{\dmin}, \frac{\dmin}{r_1}, \frac{r_1}{\dmax}, 1 \right\} \right) + 3 \\
&= c_2 \cdot \frac{\Delta_f L^{1/2}}{\epsilon^{-3/2}} + 2 \log_{\omega \SPTtwo}\left( \max\left\{ \frac{c_2 \omega}{2} \cdot \frac{\Delta_f L^{1/2}}{\epsilon^{3/2}}, \frac{\SPTtwo}{\omega c_1^{1/2}} \cdot \frac{\epsilon^{1/2}}{L^{1/2} r_1}, \frac{\beta \theta}{2 \omega} \cdot \frac{r_1 L^{1/2}}{\epsilon^{1/2}}, 1 \right\} \right) + 3
\end{flalign*}
where
$$
c_2 := \frac{4 c_1^{1/2} \omega}{\beta \theta \SPTtwo}
$$
is a problem-independent constant.
As $c_1, c_2, \omega, \beta, \theta, \SPTone, \SPTtwo$ and $\SPTthree$ are problem-independent constants (see the definition of $c_1$ in Lemma~\ref{lem:gradient-bound_distance} and the requirements of Algorithm~\ref{alg:SOAT}) the result follows.
\end{proof}

\section{Proof of Theorem~\ref{thm:main-superlinear-convergence-result}}\label{sec:proof-of-superlinear-convergence}

We first prove Theorem~\ref{thm:general-superlinear-convergence-result} and then reduce Theorem~\ref{thm:main-superlinear-convergence-result} 
to Theorem~\ref{thm:general-superlinear-convergence-result}.
The following fact will be useful.

\newcommand{\SCS}{C}

\begin{fact}[\cite{bubeck2014convex}]
If $f$ is $\alpha$-strongly convex and $\sm$-smooth on the set $\SCS$ (i.e., $\alpha \eye \preceq \grad^2 f (x) \preceq \sm \eye$ for all $x \in \SCS$) then 
\begin{flalign} 
\alpha \| x - x_\star \| \le \| \grad f(x) \| \le \sm \| x - x_\star \|
\end{flalign}
where $x_\star$ is any minimizer of $f$.
\end{fact}

\begin{theorem}\label{thm:general-superlinear-convergence-result}
Suppose that $f$ is $L$-Lipschitz, $\grad f(x_\star) = 0$ and there exists $\alpha, \sm, t > 0$ such that $\alpha \eye \preceq \grad^2 f (x) \preceq \sm \eye$ for all $x \in \{x \in \R^{n} : \| x - x_\star \| \le t \}$.
Consider the set
$$
\SCS := \left\{ x \in \R^{n} : f(x) \le f(x_\star) +  \frac{2 \eta^2}{\alpha}, \| x - x_\star \| \le \eta  \right\}
$$
with \[ 
\eta = \min\left\{ t, \frac{\alpha^3 (1 - \SPTone)}{2 L \sm^2} \min\left\{ \frac{1}{2}, \omega \SPTtwo - 1 \right\}, \frac{3 (1 - \beta) \alpha }{\left(2 + 12 (1-\beta) \SPTone c_1\right)L\omega \SPTtwo}, \frac{(1-\beta) \alpha}{2 \omega \SPTtwo \beta \theta  L c_1} \right\} 
\]
then if $x_i \in \SCS$ then for $k \ge 2 + i + \log_{\SPTtwo \omega}(\frac{\eta}{\| d_i \|})$ we have
$$
\| x_{k+1} - x_\star \| \le \frac{2L \sm^2}{\alpha^3 (1 - \SPTone)} \| x_{k} - x_\star \|^2.
$$
\end{theorem}

\begin{proof}
We begin by establishing the premise of
Lemma~\ref{lem:superlinear-big-r-k}.
First we establish $x_k \in \SCS \implies x_{k+1} \in \SCS$.
Suppose that $x_k \in \SCS$ then $f(x_{k+1}) \le f(x_k) \le f(x_\star) + \frac{2 \eta^2}{\alpha}$. By strong convexity we get $x_{k+1} \in \SCS$.

Next we establish that
$\min\{ \SPTtwo r_{k}, \| x_{k+1} - x_\star \| \} \le \| d_k \| \le  \omega \SPTtwo \| x_{k} - x_\star \|$.
By strong convexity and \eqref{eq:model-reduction-by-dist1} we have
\[
\frac{\alpha + \delta_k}{2} \| d_k \|^2 - \| \grad f(x_k) \| \| d_k \| \le \Mk(d_k) \le 0
\]
which implies $\| d_k \| \le \frac{2\| \grad f(x_k) \|}{\alpha + \delta_k}$.
Furthermore, by \eqref{eq:Lip-grad1}, \eqref{eq:model-gradient-weaker} and $\| d_k \| \le \frac{2\| \grad f(x_k) \|}{\alpha + \delta_k}$ we have
\[
\| \grad f(x_k+d_k) + \delta_k d_k \| \le \| \grad \Mk(d_k) + \delta_k d_k \| + \frac{L}{2} \| d_k \|^2 \le \SPTone \| \grad f(x_k + d_k) \| + \frac{2 L \| \grad f(x_k) \|^2}{\alpha^2}
\]
which after rearranging
\begin{flalign} 
\| \grad f(x_k+d_k) + \delta_k d_k \| \le \frac{2L}{\alpha^2 (1 - \SPTone)} \| \grad f(x_k) \|^2
\end{flalign}
By strong convexity and smoothness,
\begin{flalign}\label{superlinear-to-prox}
\| x_{k} + d_k - \hat{x}_k \| \le \frac{2L \sm^2}{\alpha^3 (1 - \SPTone)} \| x_{k} - x_\star \|^2
\end{flalign}
where $\hat{x}_k := \min f(x) + \frac{\delta_k}{2} \| x - x_k \|^2$.
Therefore, as $\| x_k - x_\star \| \le \frac{\alpha^3 (1 - \SPTone)}{2 L \sm^2} \min\left\{ \frac{1}{2}, \omega \SPTtwo - 1 \right\}$,
\[
\| x_{k} + d_k - \hat{x}_k \| \le \min\left\{ \frac{1}{2}, \omega \SPTtwo - 1 \right\} \| x_{k} - x_\star \|  
\]
which combined with the triangle inequality and $\| \hat{x}_k - x_k \| \le \| x_k - x_\star \|$ gives
\[
\| d_k \| \le \| x_k + d_k - \hat{x}_k \| + \| x_k - \hat{x}_k \| \le \omega \SPTtwo \| x_k - x_\star \|
\]
Furthermore, if $\| d_k \| < \SPTtwo r_k$ then by \eqref{eq:comp-delta-radius1} we have $\delta_k = 0$ and $\hat{x}_k = x_\star$ which gives
\begin{flalign*}
\| x_k + d_k - x_\star \| \le \frac{1}{2} \| x_k - x_\star \| \le \| x_k - x_\star \| - \| x_k + d_k - x_\star \| \le \| d_k \|.
\end{flalign*} 

Next we show $x_k \in \SCS$ implies $\hat{\rho}_k \ge \beta$.
To obtain a contradiction we assume
$\hat{\rho}_k < \beta$, by the definition of the model,
\eqref{eq:model-gradient-weaker}, strong convexity, and \eqref{eq:grad-bound-distance1} we get
\begin{flalign*} 
\Mk(d_k) &= \frac{1}{2} d_k^T \grad^2 f(x_k) d_k + \grad f(x_k)^T d_k = d_k^T (\grad^2 f(x_k) d_k + \delta_k d_k + \grad f(x_k)) - \frac{1}{2} d_k^T (\grad^2 f(x_k) + 2 \delta_k \eye ) d_k \\
&\le \SPTone \| d_k \| \| \grad f(x_{k} + d_k) \| - \frac{1}{2} d_k^T (\grad^2 f(x_k) + 2 \delta_k \eye ) d_k \\
&\le \SPTone \| d_k \| \| \grad f(x_{k} + d_k) \| - \frac{\alpha}{2} \| d_k \|^2 \\
&\le \SPTone c_1 L \| d_k \|^3 - \frac{\alpha}{2} \| d_k \|^2.
\end{flalign*}
It follows that by inequality \eqref{eq:Lip-progress1}, $\| d_k \| \le \omega \SPTtwo \| x_k - x_\star \| \le  \frac{3 (1 - \beta) \alpha }{\left(2 + 12 (1-\beta) \SPTone c_1\right)L}$, inequality \eqref{eq:grad-bound-distance1}, $\| d_k \| \le \omega \SPTtwo \| x_k - x_\star \| \le  \frac{(1-\beta) \alpha}{2 \beta \theta L c_1}$ we have
\begin{flalign*}
f(x_k) - f(x_{k+1}) &\ge  -\Mk(d_k) - \frac{L}{6} \| d_k \|^3 \\
&\ge -\beta \Mk(d_k) + (\beta - 1) \Mk(d_k)
- \frac{L}{6} \| d_k \|^3 \\
&\ge -\beta \Mk(d_k) + \frac{(1-\beta) \alpha}{2} \| d_k \|^2 +  (\beta - 1) \SPTone c_1 L \| d_k \|^3 - \frac{L}{6} \| d_k \|^3 \\
&\ge -\beta \Mk(d_k) + \frac{(1-\beta) \alpha}{2} \| d_k \|^2 -L \|d_k\| ^ 3 \left(\frac{1 + 6 (1 - \beta) \SPTone c_1}{6}\right)\\
&\ge -\beta \Mk(d_k) + \frac{(1-\beta) \alpha}{2} \| d_k \|^2 - \frac{(1-\beta) \alpha}{4} \| d_k \|^2 \\
&\ge -\beta \Mk(d_k) + \frac{(1-\beta) \alpha}{4} \| d_k \|^2 \\
&\ge -\beta \Mk(d_k) + \frac{(1-\beta) \alpha}{4 L c_1} \| \grad f(x_k + d_k) \| \\
&\ge \beta \left(-\Mk(d_k) + \frac{\theta}{2} \|\grad f(x_k + d_k) \|d_k \| \right)
\end{flalign*}

which after rearranging gives:
		
    \[
    \hat{\rho}_k = \frac{f(x_k) - f(x_k + d_k)}{-\Mk(d_k) + \frac{\theta}{2} \|\grad f(x_k + d_k)  \| d_k \|} \ge \beta
    \]

which gives our desired contradiction.

With the premise of Lemma~\ref{lem:superlinear-big-r-k} established we conclude that for $k \ge 2 + i + \log\left( \eta / \| d_i \| \right)$ we have $\delta_k = 0$ and therefore by \eqref{superlinear-to-prox} we get the desired result.
\end{proof}

The following Lemma is a standard result but we include it for completeness.

\begin{lemma}\label{lem:neigh-exists-strongly-convex}
If $\grad^2 f(x_\star)$ is twice differentiable and positive definite, then there exists a neighborhood $\NBH$ and positive constants $\alpha, \beta > 0$ such that $\alpha \eye \preceq \grad^2 f(x) \preceq \sm \eye$ for all $x \in \NBH$.
\end{lemma}
\begin{proof} 
As $\grad^2 f$ is twice differentiable and the
fact that continuous functions on 
compact sets are bounded we conclude
that there exists a neighborhood $\NBH$ around $x_\star$ that $\grad^2 f$ is $L$-Lipschitz for some constant $L \in (0,\infty)$.
Then by using the fact that there exists positive constants $\alpha', \beta' \in (0,\infty)$ \text{ s.t. } $\alpha' \eye \preceq \grad^2 f(x_\star) \preceq \beta' \eye$ we conclude for sufficiently small ball around $x_\star$ we have $\alpha' / 2 \eye \preceq \grad^2 f(x) \preceq 2 \beta' \eye$ for all $x$ in a sufficiently small neighborhood $\NBH' \subseteq \NBH$.
\end{proof}

\begin{proof}[Proof of Theorem~\ref{thm:main-superlinear-convergence-result}]
Follows by Lemma~\ref{lem:neigh-exists-strongly-convex} and Theorem~\ref{thm:general-superlinear-convergence-result}.
\end{proof}

\section{Solving trust-region subproblem}\label{sec:solving-tr-subproblem}

In this section, we detail our approach to solve the trust-region subproblem. We first attempt to take a Newton's step by checking if $\grad^2 f(x_k) \succeq 0$ and $\|\grad^2 f(x_k) ^ {-1} \grad f(x_k)\| \leq r_k$. However, if that is not the case, then the optimally conditions mentioned in \eqref{eq:subproblem-termination-criteria}, will be a key ingredient in our  approach to find $\delta$ and hence $d_k(\delta)$. Based on these optimally conditions, we will define a univariate function $\phi$ that we seek to find its root at each iteration. In our implementation we use $\SPTthree = 1.0$ for \eqref{eq:model-reduction-by-dist1}  which is the same as satisfying \eqref{eq:PSD-grad-squared-f-delta}. The function $\phi$ is defined as bellow:

$$\phi(\delta) := \begin{cases}
	-1, & \text{if $\grad^2 f(x_k) + \delta \eye \nsucceq 0 \text{ or } \|d_k (\delta) \| >  r_k $}\\
	+1, & \text{if $\grad^2 f(x_k) + \delta \eye \succeq 0 \And \|d_k (\delta) \| < \SPTtwo r_k$} \\
    0, & \text{if $\grad^2 f(x_k) + \delta \eye \succeq 0 \And \|d_k (\delta)\| \leq r_k$}
\end{cases}$$

where:
$$d_k (\delta) := (\grad^2 f(x_k) + \delta \eye) ^ {-1} (-\grad f(x_k))$$

When we fail to take a Newton's step, we first find an interval $[\delta, \delta']$ such that $\phi(\delta) \times \phi(\delta ') \leq 0$. Then we apply bisection method to find $\delta_k$ such that $\phi(\delta_k) = 0$. In case our root finding logic failed,

then we use the approach from the hard case section under chapter 4 "Trust-Region Methods" in \cite{nocedal2006numerical} to find the direction $d_k$.

\edit{The logic to find the interval $[\delta, \delta']$ is summarized as follow. We first compute $\phi(\delta)$ using the $\delta$ value from the previous iteration. Then we search for $\delta'$ by starting with $\delta' = 2 \delta$. We compute   $\phi(\delta')$ and in the case $\phi(\delta') < 0$, we update $\delta'$ to become twice its current value, otherwise if $\phi(\delta') > 0$, we update $\delta'$ to become half its current value. We keep repeating this logic until we get a $\delta'$ such that $\phi(\delta) \times \phi(\delta ') \leq 0$ or until we reach the maximum iteration limit which is marked as a failure}. 

The whole approach is summarized in Algorithm \ref{alg:TRS}:

\begin{algorithm}[H]
\If{$\grad^2 f(x_k) \succeq 0 $}{
    $d_k = - \grad^2 f(x_k) ^ {-1} \grad f(x_k)$
    
    \uIf{$ \|d_k\| \leq r$}{
        \Return $d_k$\;
    }
}
\uIf{hard case}{
Find $d_k$ using \cite[pages~87-88]{nocedal2006numerical} \;

\Return $d_k$
}
\Else{
Find initial interval $[\delta, \delta']$ using the $\phi$ function such that $\phi(\delta) \times \phi(\delta') \leq 0$ \; 

Use bisection method to find $\delta_k$ such that $\phi(\delta_k) = 0$ \; 

\Return $d_k (\delta_k)$
}
\caption{trust-region subproblems solver}
\label{alg:TRS}
\end{algorithm}

\section{Experimental results details \label{sec:experimental-full-results}}

\subsection{Learning linear dynamical systems \label{sec:experimental-learning-linear-dynamical-systems}}
The time-invariant linear dynamical system is defined by:
\begin{flalign*}
h_{t+1} &= A h_t + B u_t + \xi_t \\
x_t &= h_{t} + \vartheta_t
\end{flalign*}
where the vectors $h_t$ and $x_t$ represent the hidden and observed state of the system at time $t$. Here  $u_t, \vartheta_t \sim N(0,1 )^ d$, 

$\xi_t \sim N(0, \sigma) ^ d$ and $A \text{ and } B$ are linear transformations.

The goal is to recover the parameters of the system using maximum likelihood estimation and hence we formulate the problem as follow:
$$\min_{A,B,h} \sum_{t=1}^T \frac{\|h_{t+1} - A h_t - B u_t\| ^ 2}{\sigma ^ 2} +  \| x_t - h_t \| ^ 2$$

We synthetically generate examples with noise both in the observations and also the evolution of the system. The entries of the matrix $B$ are generated using a Normal distribution $N(0, 1)$. For the matrix $A$, we first generate a diagonal matrix $D$ with entries drawn from a uniform distribution $U[0.9,0.99]$ and then we construct a random orthogonal matrix $Q$ by randomly sampling a matrix $W \sim N(0,1)^{d \times d}$ and then performing an QR factorization. Finally using the matrices $Q$ and $D$, we define $A$:
$$
A = Q^T D Q
$$
We compare our method against the Newton trust-region method available through the Optim.jl package \cite{mogensen2018optim} licensed under \url{https://github.com/JuliaNLSolvers/Optim.jl/blob/master/LICENSE.md}. In the results/learning problem subdirectory in
the git repository, 

we present the full results of running our experiments on 60 randomly generated instances with $T = 50$, $d = 4$, and $\sigma = 0.01$ where we used a value of $10 ^ {-5}$ for the gradient termination tolerance. 

This experiment was performed on a MacBook Air (M1, 2020) with 8GB RAM.

\subsection{Matrix completion \label{sec:experimental-cmatrix-completion}}
\edit{
The original power consumption data is denoted by a matrix $D \in R ^ {n_1 \times n_2}$ where $n_1$ represents the number of measurements taken per day within a 15 mins interval and $n_2$ represents the number of days. Part of the data is missing, hence the goal is to recover the original data. The set $\Omega = \{(i, j) | D_{i,j} \text{ is observed}\}$ denotes the indices of the observed data in the matrix $D$.

We decompose $D$ as a product of two matrices $P \in R ^ {n_1 \times r}$ and $Q \in R ^ {n_2 \times r}$ where $r < n_1$ and $r < n_2$: 

\begin{flalign*}
D = P Q^T.
\end{flalign*}

To account for the effect of time and day on the power consumption data , we use a baseline estimate \cite{koren2008factorization}:

\begin{flalign*}
d_{i, j} = \mu + r_i + c_j
\end{flalign*}

where $\mu$ denotes the mean for all observed measurements, $r_i$ denotes the observed deviation during time $i$, and $c_j$ denotes the observed deviation during day $j$ \cite{zhuang2020power,koren2008factorization}.

We formulate the matrix completion problem as the regularized squared error function of SVD model \cite[Equation 10]{zhuang2020power}:

\begin{flalign*}
\min_{r,c,p,q} \sum_{(i,j)\in \Omega} (D_{i, j} - \mu - r_i - c_j - p_i q_j ^ T) + \lambda_1 (r_i ^ 2 + c_j ^ 2) + \lambda_2 (\|p_i\|_2 ^ 2 + \|q_j\|_2 ^ 2)
\end{flalign*}

We use the public data set of Ausgrid, but we only use the data from a single substation (the Newton trust-region method \cite{mogensen2018optim} is very slow for this example so testing it on all substations takes a prohibitively long time). 
We limit our option to 30 days and 12 hours measurements i.e the matrix D is of size $48 \times 30$ because with a larger matrix size, the Newton trust-region \cite{mogensen2018optim} was always reaching the iterations limit.

We compare our method against Newton trust-region algorithm available through the Optim.jl package \cite{mogensen2018optim} licensed under \url{https://github.com/JuliaNLSolvers/Optim.jl/blob/master/LICENSE.md}. In the results/matrix completion subdirectory in the git repository,

we include the full results of running our experiments on 10 instances by randomly generating the sampled measurements from the matrix $D$ with the same values for the  regularization parameters as in \cite{zhuang2020power} where we used a value of $10 ^ {-5}$ for the gradient termination tolerance. 

This experiment was performed on a MacBook Air (M1, 2020) with 8GB RAM.
}

\end{document}